\documentclass{amsart}

\usepackage[english]{babel}
\usepackage{tikz}
\usetikzlibrary{positioning,arrows}

\usepackage{amsmath,amssymb,amsthm}

\usepackage{enumitem}

\usepackage{xcolor}

\usepackage[all]{xy}

\newtheorem{theorem}{Theorem}[section]

\newtheorem{lemma}[theorem]{Lemma}

\newtheorem{question}[theorem]{Question}

\newtheorem{thmx}{Theorem}

\newtheorem{corx}[thmx]{Corollary}
\newtheorem{problem}[theorem]{Problem}	
	
\theoremstyle{definition}
\newtheorem{definition}[theorem]{Definition}

\theoremstyle{remark}
\newtheorem{example}[theorem]{Example}				
\newtheorem{remark}[theorem]{Remark}

\newcommand{\id}{\operatorname{id}}

\newcommand{\Aut}{\operatorname{Aut}}

\newcommand{\Mor}{\operatorname{Mor}}

\newcommand{\Ob}{\operatorname{Ob}}

\newcommand{\MS}{\operatorname{MaxSub_{loc}}}

\DeclareMathOperator\CR{CR}

\newcommand{\Z}{{\mathbb Z}}
\newcommand{\Pb}{{\mathbb P}}
\newcommand{\Wb}{\mathbb{W}}
\renewcommand{\Mc}{{\mathcal M}}

\newcommand{\Nc}{{\mathcal N}}

\newcommand{\Mb}{{\mathbb M}}
\newcommand{\N}{{\mathbb N}}

\newcommand{\C}{\mathcal{C}}

\newcommand{\X}{{\mathcal X}}

\newcommand{\Graphs}{{{\mathcal G}raphs}}
\newcommand{\Groups}{{{\mathcal G}roups}}
\newcommand{\Part}{{{\mathcal P}art}}
\newcommand{\SSet}{{{\mathcal S}Sets}}
\newcommand{\Sym}{{{\mathcal S}ym}}
\newcommand{\InjPart}{{{\mathcal I}njPart}}
\newcommand{\DecGraphs}{{{\mathcal D}ecGraphs}}
\newcommand{\FinDecGraphs}{{{\mathcal F}inDecGraphs}}
\newcommand{\Db}{\mathbb{D}}
\newcommand{\Fc}{\mathcal{F}}
\newcommand{\Hc}{\mathcal{H}}

\usepackage[pdfpagelabels]{hyperref}

\makeatletter
\@namedef{subjclassname@2020}{%
  \textup{2020} Mathematics Subject Classification}
\makeatother

\AtBeginDocument{%
   \def\MR#1{}
}

\begin{document}

\title[Path partial groups]{Path partial groups}
\author{Antonio D{\'\i}az Ramos}
\address{Departamento de \'Algebra, Geometr{\'\i}a y Topolog{\'\i}a, Universidad de M{\'a}laga, 29071-M{\'a}laga, Spain}
\email{adiazramos@uma.es\\
ORCID: \href{https://orcid.org/0000-0002-1669-1374}{0000-0002-1669-1374}}
\author{R\'emi Molinier}
\address{Univ. Grenoble Alpes, CNRS, IF, 38000 Grenoble, France}
\email{remi.molinier@univ-grenoble-alpes.fr\\
ORCID: \href{https://orcid.org/0000-0002-3742-5307}{0000-0002-3742-5307}}
\author{Antonio Viruel}
\address{Departamento de \'Algebra, Geometr{\'\i}a y Topolog{\'\i}a, Universidad de M{\'a}laga, 29071-M{\'a}laga, Spain}
\email{viruel@uma.es\\
ORCID: \href{https://orcid.org/0000-0002-1605-5845}{0000-0002-1605-5845}}
\thanks{This work was partially supported by MCIN/AEI/10.13039/501100011033 [PID2020-118753GB-I00 to first and third authors, and PID2023-149804NB-I00 to third author].}



\begin{abstract}
It is well known that not every finite group arises as the full automorphism group of some group. Here we show that the situation is dramatically different when considering the category  of partial groups, $\Part$, as defined by Chermak: given any group $H$ there exists infinitely many non isomorphic partial groups $\Mb$ such that $\Aut_\Part(\Mb)\cong H$. To prove this result, given any simple undirected graph $G$ we construct a partial group $\Pb(G)$, called the path partial group  associated  to $G$, such that $\Aut_\Part\big(\Pb(G)\big)\cong \Aut_\Graphs(G)$.
\end{abstract}

\maketitle

\section{Introduction}

From the seminal work of Galois to the discovery of sporadic simple groups, describing abstract groups as the full automorphisms group of a mathematical object can be considered the source of Group Theory, and it certainly sits in the core of Representation Theory. Within this framework, it is then natural to ask whether for a fixed category $\C$, and given an abstract group $H$, there exists an object $X$ in $\C$ such that $H\cong\Aut_\C(X)$. If any group (resp.\ finite group) $H$ can be so represented, the category $\C$ is said to be universal (resp.\ finitely universal) \cite[Section 4.1]{Babai}.

Some categories have already been shown to be either universal or finitely universal. For example, $\Graphs$, the category of simple undirected graphs, is known to be universal \cite{Fr39, Gr, Sa}, and this has been used to show that algebras over any field are finitely universal \cite{CLTV}. In addition, the homotopy category of rational elliptic spaces is universal \cite{CV}, and the category of compact hyperbolic $n$-manifolds and their isometries is finitely universal \cite{BeLo2005} for any $n\geq 2$. However, in general identifying universal categories is a hard task and, for instance, deciding whether Galois extensions over the field of rational numbers is finitely universal, the Inverse Problem of Galois Theory \cite{Serre_Galois}, is still an open question. A more comprehensive survey of this topic may be found in \cite{Babai} and \cite[Introduction]{jones}.

Ironically, the category $\Groups$ is not even finitely universal; it is well known that  a nontrivial cyclic group of odd order cannot arise as the full automorphism group of any group.  So it is natural to ask whether it is possible to ``sensibly enlarge" the category $\Groups$ to a universal category $\C$. More precisely

\begin{question}\label{core_question}
Does there exist a universal category $\C$ such that $\Groups$ is fully embedded in $\C$?
\end{question}

We tackle this question by considering $\Part$, the category of partial groups (see Section \ref{sec:partials} for definitions). Partial groups, defined by  Chermak \cite[Section 2]{Ch0}, generalise the concept of group and  are introduced as a  setting for the study of the $p$-local structure of finite groups within the framework of fusions systems as defined by Broto-Levi-Oliver \cite{BLO} (see also the monographies \cite{AKO} and \cite{Craven}). Another example of partial group is the set $M_{13}$, which appears naturally in the context of the Mathieu group $M_{12}$, see \cite{CEM2006} and \cite{MR3471269}. In addition, every group gives rise to a partial group in a natural way  \cite[Example 2.4.(1)]{Ch0} and so $\Groups$ is fully embedded in $\Part$. Therefore we answer Question \ref{core_question} in the positive by proving

\begin{thmx}[Theorem \ref{thm:theorem_A}]\label{thm:introduction_thmA}
The category $\Part$ is universal. Moreover, given an abstract group $H$ there exist infinitely many non isomorphic partial groups $\Mb$ such that $\Aut_\Part(\Mb)\cong H$.
\end{thmx}

The category of monoids, and thus $\Groups$, can be fully embedded into the category $\SSet$ of simplicial sets by sending a monoid $M$, eventually a group, to $\Wb(M)$ the graded set of all possible finite words and where faces and degeneracies are induced by the product of two consecutive elements (or dropping the first or last element) and injecting the identity somewhere respectively. In the same way, Broto and Gonz{\'a}lez \cite[Section 2]{alex-carles} highlighted that there is a fully faithful functor, which extends the latter, from the category of partial monoids (see Definition \ref{def:PG}), and thus $\Part$, to $\SSet$ which sends a partial monoid $\Mc$, eventually a partial group, to its domain $\Db\subseteq \Wb(\Mc)$, i.e. the graded sets of all words that can be multiplied. This gives another proof of the universality of simplicial sets which can already be deduced from the universality of the category of monoids (see for example \cite[Proposition 5]{monoids}) and the above functor. Actually the involution induced by the inverse in a partial group gives a more richer structure to its domain where even non-order preserving maps can act: it is a symmetric simplicial set \cite{Grandis}. This gives a fully faithful embedding of $\Part$ into the category $\Sym$ of symmetric simplicial sets \cite[Theorem A]{HL} and we can therefore deduce that  $\Sym$ is also universal.

\begin{corx}
The category $\Sym$ of symmetric simplicial set is universal. Moreover, given an abstract group $H$ there exist infinitely many non isomorphic symmetric simplicial sets $X$ such that $\Aut_\Sym(X)\cong H$.
\end{corx}

The proof of Theorem \ref{thm:introduction_thmA} is built upon a functor, defined in Section \ref{sec:graphs}, that maps every pair $(G,\Hc)$, for $G=(V,E)$ a simple undirected graph and $\Hc=\{H_v\}_{v\in V}$ a collection of non-trivial groups, to a partial group $\Mb(G,\Hc)$ such that the group of automorphisms $\Aut_\Part\big(\Mb(G,\Hc)\big)$ is closely related to the groups $\Aut_\Graphs(G)$ and $\prod_{v\in V} \Aut_\Groups(H_v)$  (see Theorem \ref{thm:aut_of_partial}). The algebraic structure of the partial group $\Mb(G,\Hc)$ is related to path concatenation within $G$. In fact, the special case with $\Hc$ consisting of copies of $\Z_2$, i.e., the cyclic group of order $2$, is the central character in this work: path partial groups.

\begin{definition}\label{def:pathpartialgroup}
Let $G$ be a simple undirected graph. The \emph{path partial group} associated to the graph $G$, denoted by  $\Pb(G)$, is
the partial group $\Mb(G,\{\Z_2\}_{v\in V}).$
\end{definition}

The algebraic structure of $\Pb(G)$ can be described in terms of the paths contained in the cliques of $G$ and path concatenation, together with the fact that vertices behave as inverses of themselves, see discussion after Lemma \ref{lem:graphs_embeds}. Its construction resembles the ideas underlying path algebras, i.e., the algebra associated to a quiver \cite[II.1.2]{assem_skowronski_simson_2006}. The following result is the key ingredient in the proof of Theorem \ref{thm:introduction_thmA}.

\begin{thmx}[Theorem \ref{thm:aut_of_partial_Z2}]\label{thm:introduction_thmC}
Let $G$ be a simple undirected graph. Then $$\Aut_\Part\big(\Pb(G)\big)\cong\Aut_\Graphs(G).$$
\end{thmx}

In general, the graph $G$ can be recovered up to isomorphism from the partial group $\Mb(G,\Hc)$ whenever the collection $\Hc$ consists of locally finite groups (see Theorem \ref{thm:maximalsub}). In particular, the path partial group associated to a graph $G$ determines the isomorphism type of $G$, that is, the path partial group is a classifying invariant for graphs. The situation for path algebras is similar: a quiver can be recovered up to isomorphism from its path algebra \cite[II.3.6]{assem_skowronski_simson_2006} and, in addition, some invariants of the quiver correspond to invariants of its associated path algebra \cite[II]{assem_skowronski_simson_2006}. Therefore, it makes sense to propose the following problem.

\begin{problem}\label{problem:proposed}
Describe graph theoretical invariants of a simple undirected graph in terms of algebraic invariants of its associated path partial group.
\end{problem}

Besides the aforementioned setup of quivers and their associated path algebras, the kind of study proposed in Problem \ref{problem:proposed} is well established in other fields as directed graphs and their associated Leavitt path algebras, see \cite{leavitt} or the Simplicity Theorem \cite[2.9.1]{leavitt} as a concrete instance of interaction between the graph side and its algebraic counterpart. In other areas as  finite simplicial graphs and their associated right-angled Artin groups, studying this type of relations is also an active line of work, see \cite[Problem 1.1]{ramon}. In that work, the authors add the graph theoretical property of $k$-colorability to the list of properties of simplical graphs that can read off from the algebraic properties of its associated right-angled Artin group. This list includes, among others, the graph theoretical properties of being a join or disconnected, and these properties correspond to the algebraic properties of decomposing as a non-trivial direct product or decomposing as a non-trivial free product respectively. See \cite[p.\ 2]{ramon} for more details.

\subsection*{Outline of the paper:} We start giving in Section \ref{sec:partials} basic notions about partial groups and some related constructions as colimit of partial groups. This preliminaries are the necessary background to construct the partial group $\Mb(G,\Hc)$ associated to a decorated simple undirected graph $(G,\Hc)$ in Section \ref{sec:graphs}. Then in Section \ref{section:embedding} we deepen into this construction and show how to embed the category of decorated graphs into the category of partial groups. In Section \ref{sec:path_partial}, we restrict the study to path partial groups and prove Theorems \ref{thm:introduction_thmA} and \ref{thm:introduction_thmC}. We also include in this section a digress on rigidity in the categories of groups and partial groups. Finally, in Section \ref{sec:maximalsub}, we begin the study of Problem \ref{problem:proposed}.

\subsection*{Acknowledgements:} The authors are thankful to  Ellen Henke for showing interest in these results and pointing out reference \cite{MR4548989}, and to Edoardo Salati for providing us with an earlier version of his work.


\section{Basics on partial groups}\label{sec:partials}

The notion of partial group is due to Chermak. Here we introduce the basic definitions and some useful properties that are needed in the following sections. More details on the subject can be found in \cite[Section 2]{Ch0} or in \cite[Section 1]{Ch1}.

For $X$ a set, $\Wb(X)$ will denote the free monoid on $X$ and for two words $u,v\in\Wb(X)$, $u\circ v$ will denote the concatenation of $u$ and $v$. We also identify $X$ as the subset of words of length 1 in $\Wb(X)$. Finally, given two sets $X$ and $Y$ and a map $\varphi\colon X\to Y$, we will denote by $\overline{\varphi}\colon \Wb(X)\to\Wb(Y)$ the map induced by $\varphi$ defined as follows: for $u=(x_1,x_2,\dots,x_n)\in\Wb(X)$, $\overline{\varphi}(u)=(\varphi(x_1),\varphi(x_2),\dots,\varphi(x_n))$.

\begin{definition}\label{def:PG}
Let $\Mc$ be a non-empty set and $\Db\subseteq \Wb(\Mc)$ be a subset such that,
\begin{enumerate}[label=(D\arabic*)]
\item\label{cond:D1} $\Mc\subseteq \Db$; and
\item\label{cond:D2} $u\circ v\in\Db\Rightarrow u,v\in\Db$ (in particular, $\emptyset\in\Db$).
\end{enumerate}
A mapping $\Pi:\Db\rightarrow\Mc$ is a \emph{product} if
\begin{enumerate}[label=(P\arabic*)]
\item\label{cond:P1} $\Pi$ restricts to the identity on $\Mc$; and
\item\label{cond:P2} if $u\circ v\circ w\in \Db$ then $u\circ \Pi(v)\circ w\in\Db$ and
\[\Pi(u\circ v\circ w)=\Pi\left(u\circ\Pi(v)\circ w\right).\]
\end{enumerate}
The \emph{unit} is then defined as $\Pi(\emptyset)$ and we will denote it by $1_\Mc$ or $1$ when there is no ambiguities.
A \emph{partial monoid} is a triple $(\Mc,\Db,\Pi)$ where $\Pi$ is a product defined on $\Db$ and $\Db$ is called the \emph{domain} of $(\Mc,\Db,\Pi)$.

An \emph{inversion} on $\Mc$ is an involutory bijection $x\mapsto x^{-1}$ on $\Mc$ together with the induced mapping  $u\mapsto u^{-1}$ on $\Wb(\Mc)$ defined by,
\[u=(x_1,x_2,\dots,x_n)\mapsto (x_n^{-1},x_{n-1}^{-1},\dots,x_1^{-1}).\]
A \emph{partial group} is a tuple $\Mb=\left(\Mc,\Db,\Pi,(-)^{-1}\right)$ where $\left(\Mc,\Db,\Pi\right)$ is a partial monoid and $(-)^{-1}$ is an inversion on $\Mc$ satisfying
\begin{enumerate}[resume,label=(P\arabic*)]
\item\label{cond:P3} if $u\in\Db$ then $u^{-1}\circ u\in\Db$ and $\Pi(u^{-1}\circ u)=1$.
\end{enumerate}

\end{definition}

A word $w\in\Db$ will be called \emph{non-degenerate} if there is no $1$ in $w$. Then, $\Db$ is totally determined by its non-degenerate words.

\begin{remark}
The notion of locality \cite[Definition 2.9]{Ch0}, a particular type of partial group, plays a central role in Chermak's work and the study of $p$-local structure of finite groups. One can check that the partial groups considered in this work (Definitions \ref{def:M(C)} and \ref{def:DM}) are not localities in general as they already fail to be \emph{objective} partial groups (\cite[Definition 2.6]{Ch0}).
\end{remark}

\begin{example}\label{ex:grp}
 Let $\Mb=\left(\Mc,\Db,\Pi,(-)^{-1}\right)$ be a partial group.  If $\Db=\Wb(\Mc)$ then $\Mb$ is a group via the binary operation $(x,y)\in\Mc^2\mapsto \Pi(x,y)\in\Mc$. Moreover, if $H$ is a group and $\Pi\colon \Wb(H)\to H$ is the multivariate product induced by the binary product on $H$, $(H,\Wb(H),\Pi,(-)^{-1})$ is a partial group.
\end{example}

\begin{example}\label{ex:freepargp}
Let $\Fc(a)=\{1,a,a^{-1}\}$ and set the non-degenerate words of $\Db_a$ to be all possible words alternating $a$'s and $a^{-1}$'s. In other words, the non-degenerate words of $\Db_a$ are all the different finite sub-words of the infinite word $(a,a^{-1},a,a^{-1},a,a^{-1},\dots)$. The inversion is understood  and, for any word $u\in\Db\left(\Fc(a)\right)$,
\[
\Pi(u)=
\begin{cases}
1 &\text{if the number of $a$'s equal the number of $a^{-1}$'s},\\
a &\text{if the number of $a$'s exceed the number of $a^{-1}$'s (necessarily by 1)},\\
a^{-1} &\text{if the number of $a^{-1}$'s exceed the number of $a$'s (necessarily by 1)}.
\end{cases}
	\]
One can then check that $\left(\Fc(a),\Db_a,\Pi_a,(-)^{-1}\right)$ defines a partial group.
\end{example}

This last example is actually the \emph{free partial group on the set $\{a\}$} as detailed in \cite[Lemma 1.12]{Ch1}.

\begin{definition}\label{def:partial_subgroup_and_subgroup}
Let $\Mb=\left(\Mc,\Db,\Pi,(-)^{-1}\right)$ be a partial group.
A \emph{partial subgroup} of $\Mb$ is a subset $\Nc$ of $\Mc$ closed under inversion and such that $\Pi\left(\Db\cap\Wb(\Nc)\right)\subseteq \Nc$. Its partial group structure is given by $\left(\Nc, \Db\cap\Wb(\Nc),\Pi,(-)^{-1}\right)$. If $\Db\cap\Wb(\Nc)=\Wb(\Nc)$ then we say that $\Nc$ is a \emph{subgroup} of $\Mb$.
\end{definition}

\begin{definition}
Let $\Mb_1=\left(\Mc_1,\Db_1,\Pi_1,(-)^{-1}\right)$ and $\Mb_2=\left(\Mc_2,\Db_2,\Pi_2,(-)^{-1}\right)$ be partial groups.
A map of sets $\varphi\colon \Mc_1\to \Mc_2$ is a \emph{homomorphism of partial groups} from $\Mb_1$ to $\Mb_2$ if
\begin{enumerate}[label=(H\arabic*)]
\item $\overline{\varphi}\left(\Db_1\right)\subseteq \Db_2$;
\item for any $u\in\Db_1$, $\Pi_2\left(\overline{\varphi}(u)\right)=\varphi\left(\Pi_1(u)\right)$.\label{equ:varphioverlinevarphi_commute}
\end{enumerate}
The \emph{kernel} of $\varphi$ is the partial subgroup of $\Mc_1$ given by
\[
\ker(\varphi)=\{x\in \Mc_1\text{ $|$ }\varphi(x)=1_{\Mc_2}\}.
\]
A homomorphism $\varphi\colon \Mc_1\to \Mc_2$ is an \emph{isomorphism} of partial groups if the map $\varphi$ is bijective and $\varphi^{-1}$ is also a morphism of partial groups. Finally, an \emph{automorphism} of the partial group $\Mc_1$ is an \emph{isomorphism} $\varphi\colon \Mc_1\to\Mc_1$.
\end{definition}

\begin{lemma}[{\cite[Lemma 1.15]{Ch1}}]\label{lem:subgroup_to_subgroup}
Consider  partial groups $\Mb_1=\left(\Mc_1,\Db_1,\Pi_1,(-)^{-1}\right)$ and $\Mb_2=\left(\Mc_2,\Db_2,\Pi_2,(-)^{-1}\right)$ and a homomorphism of partial groups between them, $\varphi\colon \Mc_1\to \Mc_2$. Then, if $\Nc$ is a subgroup of $\Mb_1$, $\varphi(\Nc)$ is a subgroup of $\Mb_2$.
\end{lemma}

The category with objects the partial groups and equipped with the notion of homomorphism above and the usual composition of maps is denoted $\Part$. In particular, given a partial group $\Mc$, the set of all its automorphisms is a group that we will denote $\Aut_\Part(\Mc)$. Notice that $\Part$ contains the category $\Groups$ of groups as a full subcategory as illustrated in Example \ref{ex:grp}. We are here interested in the following question.

\begin{question}\label{question}
Is the category $\Part$ universal? That is, given an arbitrary group $H$, does there exist a partial group $\Mb$ such that $\Aut_\Part(\Mb)=H$?
\end{question}

The category $\Part$ has all limits and all colimits, and here we only discuss colimits as we will employ them later. See \cite[Appendix A]{Ch1} or \cite{MR4548989} for constructions of limits and colimits (One can also construct limits and colimits through symmetric sets \cite[Corollary C]{HL}). So let $\C$ be a small category and let $\Mb\colon \C\to \Part$ be a functor, $C\mapsto \Mb(C)=(\Mc(C),\Db(C),\Pi(C),(-)_C^{-1})$. Then the colimit in the category $\Part$,
\[
\mathop{colim}\limits_{C\in \C} \Mb(C)=(\Mc,\Db,\Pi,(-)^{-1}),
\]
has alphabet $\Mc$ equal to the quotient set
\[
\Mc=\cup_{C\in \C} \Mc(C)/\sim.
\]
Here, $\sim$ is the smallest equivalence relation that contains the following relation,
\begin{equation}\tag{C1}\label{equ:part_colim_1}
f\in \Mor_\C(C_1,C_2), x\in \Mc(C_1)\Rightarrow x\sim \Mc(f)(x),
\end{equation}
and such that there is a map $\Pi$ closing the following diagram, 
\begin{equation}\tag{C2}\label{equ:part_colim_2}
\begin{gathered}
\xymatrix@R=15pt{
\cup_{C\in \C} \Db(C)\ar[d]\ar[rrr]^{\cup_{C\in \C} \Pi(C)}&&& \cup_{C\in C} \Mc(C)\ar[d]\\
\Db\ar[rrr]^\Pi &&& \Mc.
}
\end{gathered}
\end{equation}
In this diagram, $\Db$ is defined as the subset of words $u$ in the free monoid $\Wb(\Mc)$ for which there exists an object $C\in \Ob(\C)$ and a word $v$ in $\Db(C)$ such that $u$ is the image of $v$ under the component-wise application of $\sim$, the rightmost vertical arrow is induced by $\sim$, and the leftmost vertical arrow by its component-wise application.

We denote the equivalence classes corresponding to $\sim$ with brackets, $[-]$. Inversion on $\Mc$ is the only map making commutative the following diagram, where again vertical arrows are induced by $\sim$,
\[
\xymatrix@R=15pt{
\cup_{C\in \C} \Mc(C)\ar[d]\ar[rr]^{\cup_{C\in \C} (-)_C^{-1}}&& \cup_{C\in C} \Mc(C)\ar[d]\\
\Mc\ar[rr]^{(-)^{-1}} && \Mc.
}
\]


\section{Partial groups out of simple graphs}\label{sec:graphs}

Our goal in this section is to construct a partial group out of a graph decorated with a group on each vertex. Given a collection of groups $H_1,\dots,H_n$, we start by considering its free product $H=\ast_{i=1}^{n} H_i$. If we set $H_i^*=H_i\smallsetminus\{1\}$, then the group $H$ itself may be identified with the subset of words in $\Wb(H_1^*\cup H_2^*\cup\cdots\cup H_n^*)$  which are \emph{reduced}.

\begin{definition}\label{def:reduced_word}
A word $u=(g_1,g_2,\dots,g_r)\in \Wb(H_1^*\cup H_2^*\cup\cdots\cup H_n^*)$ is \emph{reduced} if it is the empty word or if $g_j\in H_{i_j}$ then $i_j\neq i_{j+1}$, $j=1,\ldots,r-1$. Given a reduced word $u=(g_1,g_2,\dots,g_r)\in \Wb(H_1^*\cup H_2^*\cup\cdots\cup H_n^*)$, its \emph{length} is defined as $\vert u\vert=r$.
\end{definition}

By \cite[Theorem 21]{Cohen89}, there is a map
\begin{equation}\tag{$\Pi$}\label{equ:reduced_form}
\Pi\colon \Wb(H_1^*\cup H_2^*\cup\cdots\cup H_n^*)\to \Wb(H_1^*\cup H_2^*\cup\cdots\cup H_n^*)
\end{equation}
that takes any word $w$ to its unique reduced form $\Pi(w)$. We will employ this notation throughout this section. The product in $H$ corresponds then to concatenation $\circ$ followed by $\Pi$.

\begin{definition}\label{def:cyclically_reduced_word}
A word $u=(g_1,g_2,\dots,g_r)\in \Wb(H_1^*\cup H_2^*\cup\cdots\cup H_n^*)$ is \emph{cyclically reduced} if all of its cyclic permutations $(g_2,g_3,\ldots,g_r,g_1),\ldots,(g_r,g_1,\ldots,g_{r-1})$ are reduced \cite[p.\ 34]{Cohen89}.
\end{definition}

The inverse of a reduced or cyclically reduced word is also a reduced or cyclically reduced word respectively. One may also notice that if $u$ is a reduced word, then  $|\Pi(u\circ u)|\leq 2|u|$ and equality holds if and only if $u$ is cyclically reduced.

\begin{definition}
The category  $\Graphs$ is the category with objects the simple undirected graphs and morphisms the graphs homomorphisms.
\end{definition}

In what follows, $G=(V,E)$ is an object in $\Graphs$ and $\mathcal{H}=\{H_v\}_{v\in V}$ is a collection of groups. Our aim is to define a partial group depending on $G$ and $\mathcal{H}$. Now, for a clique $C\subseteq V$ in $G$, we consider the free product $H_C=\mathop{\ast}\limits_{v\in C} H_v$ and the following construction.

\begin{definition}\label{def:M(C)}
For a clique $C$ in $G$, consider the $4$-tuple,
\[
\Mb(C)=(\CR(C),\Db(C),\Pi(C),(-)_C^{-1}),
\]
where,
\begin{enumerate}
\item $\CR(C)$ is the set of cyclically reduced words in $\Wb\left(\bigcup_{v\in C} H_v^*\right)$,
\item $\Db(C)$ is the set of words $w=(u_1,u_2,\dots,u_n)\in\Wb(\CR(C))$ such that,
\begin{equation}\tag{H}\label{eq:H}
\text{for all }1\leq i \leq j \leq n,\qquad \Pi(u_i\circ u_{i+1}\circ\cdots\circ u_j)\in \CR(C),
\end{equation}
\item $\Pi(C)\colon \Db(C)\to \CR(C)$ is given by reduced form,
\[
(u_1,u_2,\dots,u_n)\in\Db(C)\mapsto \Pi(u_1\circ u_2\circ\cdots\circ u_n)\in \CR(C)\text{, and}
\]
\item inversion $(-)_C^{-1}\colon \CR(C)\to \CR(C)$ is given by inversion in the group $H_C$ restricted to $\CR(C)$.
\end{enumerate}

\end{definition}

Although the graph $G$ may have infinite order, the words we are considering here are always of finite length, and thus involve only a finite number of vertices.

\begin{theorem}\label{thm:Mb(C)_is_partial_group}
Let $G=(V,E)$ be a simple undirected graph, $\mathcal{H}=\{H_v\}_{v\in V}$ a collection of groups, and $C$ a clique in $G$. Then $\Mb(C)$ is a partial group.
\end{theorem}

\begin{proof}
By construction, \ref{cond:D1}, \ref{cond:D2} are both satisfied. Since elements in $\Mc(C)$ are cyclically reduced words, they are equal to their reduced forms and \ref{cond:P1} is satisfied. The axiom \ref{cond:P2} is satisfied thanks to \eqref{eq:H}.
Finally, to check \ref{cond:P3}, let $w=(u_1,u_2,\dots,u_n)\in \Db(C)$ and set
\[(x_1,x_2,\dots,x_{2n}):=w\circ w^{-1}=(u_1,u_2,\dots,u_n,u_n^{-1},u_{n-1}^{-1},\dots,u_1).\]
 Choose $1\leq i\leq j\leq 2n$ and let $\chi_{i,j}=(x_i,x_{i+1},\dots,x_j)$. Then, four distinct situations may happen,
\begin{enumerate}[label = (\alph*)]
\item $i\leq j\leq n$ and $\chi_{i,j}$ is a subword of $w$,
\item $n+1\leq i\leq j$ and $\chi_{i,j}$ is a subword of $w^{-1}$,
\item $i\leq n <j$ and $n-i+1\geq j-n$ and $\Pi(\chi_{i,j})=\Pi(u_i,u_{i+1},\dots, u_{2n-j})$,
\item $i\leq n <j$ and $n-i+1 \leq j-n$ and $\Pi(\chi_{i,j})=\Pi(u_{i-1}^{-1},\dots,u_{2n+1-j}^{-1})$.
\end{enumerate}
Therefore, in all cases $\Pi(\chi_{i,j})=\Pi(\alpha)$ where $\alpha$ is a subword of either $w$ or $w^{-1}$, and since both $w,w^{-1}\in\Db(C)$, \ref{cond:P3} is satisfied by condition \eqref{eq:H}.
\end{proof}

Let $\Delta$ be the poset of cliques $C\subseteq V$ in $G$. An inclusion $C_1\subseteq C_2$ between cliques gives rise to an inclusion map $\CR(C_1)\subseteq \CR(C_2)$ and to its associated map of partial groups $\Mb(C_1)\to \Mb(C_2)$. This way, we have a functor $\Mb(-)\colon \Delta\to \Part$.

\begin{definition}\label{def:DM}
Let $G=(V,E)$ be a simple undirected graph and $\mathcal{H}=\{H_v\}_{v\in V}$ a collection of groups. We define the partial group associated to the pair $(G,\mathcal{H})$ by
\[
\Mb(G,\mathcal{H})=(\Mc(G,\mathcal{H}),\Db(G,\mathcal{H}),\Pi(G,\mathcal{H}),(-)_{G,\mathcal{H}}^{-1}):=\mathop{colim}\limits_{C\in \Delta} \Mb(C).
\]
\end{definition}

Then, from the simple shape of the category $\Delta$ and the functor $\Mb(-)$, together with uniqueness of the reduced form \eqref{equ:reduced_form}, it follows that map
\begin{equation}\tag{S}\label{eq:S}
\begin{gathered}
\xymatrix@R=5pt{
\Mc(G,\mathcal{H})\ar[r] & \{\text{ cyclically reduced words in }\Wb(\cup_{v\in V} H_v^*)\text{ }\}\\
[(g_1,g_2,\ldots,g_r)]\ar@{|->}[r] & (g_1,g_2,\ldots,g_r)
}
\end{gathered}
\end{equation}
is well defined and injective, where $(g_1,g_2,\ldots,g_r)\in \CR(C)$ for some clique $C$. From now on, we will make the identification given by \eqref{eq:S} and the corresponding one for $\Db(G,\mathcal{H})$ without further notice.

\begin{lemma}\label{lem:1}
Let $G=(V,E)$ be a simple undirected graph and $\mathcal{H}=\{H_v\}_{v\in V}$ be a collection of  groups.
\begin{enumerate}[label=(\alph*)]
\item\label{lem:1a} For all $v\in V$, $H_v$ is a  subgroup of $\Mb(G,\mathcal{H})$.
\item\label{lem:1b} If $K\subseteq \Mc(G,\mathcal{H})$ is a non-trivial finite subgroup of  $\Mb(G,\mathcal{H})$, then there exists a unique $v\in V$ such that $K\leq H_v$.
\item\label{lem:1c} For all $v,v'\in V$ $v\ne v'$, and all $(h_v,h_{v'})\in H_v^*\times H_{v'}^*$, $((h_v),(h_{v'}))\in \Db(G,\mathcal{H})$ if and only if $\{v,v'\}\in E$.
\end{enumerate}
\end{lemma}

\begin{proof}
Property \ref{lem:1a} is a consequence of Equation \eqref{eq:S}. For point \ref{lem:1b}, consider a cyclically reduced word $u\in K\smallsetminus \{1\}$. If length $\vert u\vert>1$,  then the length of the cyclically reduced word
\[
\Pi(\overbrace{u\circ u \ldots \circ u}^\text{$n$ times})
\]
is $n\vert u\vert$. Therefore, if $u\in \Mc(G,\Hc)$ is an element of finite order, $u$ is a word of length 1 and it is an element of one of the groups $H_v$ for $v\in V$. Now, assume there exist $k,k'\in K\smallsetminus\{1\}$ and $v,v'\in V$ such that $k\in H_v^*$, $k'\in H_{v'}^*$ but $v\neq v'$. Then $(k,k')$ is a cyclically reduced word of length 2 in $\Wb(\cup_{v\in V} H_v^*)$ and, as $K$ is a subgroup of $\Mc(G,\Hc)$, it belongs to $K$ and $\Mc(G,\Hc)$. In particular, it cannot be an element of finite order in $\Mc(G,\Hc)$.

For part \ref{lem:1c}, assume first that $((h_v),(h_{v'}))\in \Db(G,\mathcal{H})$. Then, by the definition of colimit of partial groups, there exists a clique $C$ of $G$ and a word $(u_1,u_2)\in \Db(C)$ such that $((h_v),(h_{v'}))=(u_1,u_2)$ with $u_i\in \CR(C)$. Thus, we must have $u_1=(h_v)$ and $v\in C$, and $u_2=(h_{v'})$ and $v'\in C$. In particular, $\{v,v'\}\in E$. 
The reversed implication is straightforward.
\end{proof}

Now we define the category of decorated graphs, which is built on top of the category $\Graphs$.

\begin{definition}\label{def:decgraphs}
The category $\DecGraphs$ is the category with object set the pairs $(G,\mathcal{H})$, where $G=(V,E)\in Ob(\Graphs)$ and $\mathcal{H}=\{H_v\}_{v\in V}$ is a collection of groups indexed by the vertices of $G$, and with morphisms the pairs
\[
(f_G,f_\mathcal{H})\colon (G,\mathcal{H})\to (G',\mathcal{H}'),
\]
where $f_G\in \Mor_\Graphs(G,G')$ and $f_\Hc=\{f_v\}_{v\in V}$ is a collection of \emph{injective} group homomorphisms $f_v\colon H_v\to H'_{f_G(v)}$ for all vertices $v$ of $G$.
\end{definition}

A morphism $(f_G,f_\mathcal{H})$ in $\DecGraphs$ takes the clique $C=\{v_1,\ldots,v_n\}$ in $G$ to the clique $C'=\{f_G(v_1),\ldots,f_G(v_n)\}$ in $G'$, and note that the restriction ${f_G}_{|C}$ is injective.  In turn, it takes the cyclically reduced word
\[
u=(g_1,g_2,\dots,g_r)\in\CR(C)\text{ with $g_j\in H_{v_{i_j}}$,}
\]
to the cyclically reduced word
\[
u'=(f_{v_{i_1}}(g_1),f_{v_{i_2}}(g_2),\dots,f_{v_{i_r}}(g_r))\in\CR(C')\text{ with $f_{v_{i_j}}(g_j)\in H'_{f_G(v_{i_j})}$.}
\]
To ensure that $u'$ is a cyclically reduced word we do need that the morphisms   $f_{v_{i_1}},f_{v_{i_2}},\ldots,f_{v_{i_r}}$ are injective group homomorphisms. It is easy to check that the map defined above,
\begin{equation}\tag{V}\label{equ:varphi}
\varphi_{(f_G,f_\mathcal{H})}\colon \CR(C)\to \CR(C'),
\end{equation}
induces a homomorphism of partial groups from $\Mb(C)$ to $\Mb(C')$. Moreover, we have a natural transformation between the functors
\[
\Mb(-)\colon \Delta\to \Part\qquad\text{and}\qquad \Mb(-)'\colon \Delta'\to \Part
\]
associated to $(G,\mathcal{H})$ and $(G',\mathcal{H}')$ respectively. Upon taking colimits, we obtain a map of partial groups from $\Mb(G,\mathcal{H})$ to $\Mb(G',\mathcal{H}')$, see Definition \ref{def:DM}. This way we obtain a functor,
\begin{equation}\tag{M}\label{eq:M}
\Mb(-,-)\colon\DecGraphs\to\Part,
\end{equation}
that we will investigate below.

\section{Graph morphisms from partial groups homomorphisms}\label{section:embedding}

To further study the functor \eqref{eq:M}, we consider the following categories.

\begin{definition}\label{def:findecgraphs}
The category $\FinDecGraphs$ is the full subcategory of the category $\DecGraphs$ with objects the pairs $(G,\mathcal{H})$, where $\mathcal{H}:=\{H_v\}_{v\in V}$ is a collection of \emph{non-trivial finite} groups.
\end{definition}

\begin{definition}\label{def:injpart}
The category $\InjPart$ is the wide subcategory of the category $\Part$ with morphisms the partial group homomorphisms whose kernel has no torsion, i.e., it contains no non-trivial finite subgroup.
\end{definition}

We show below that the functor \eqref{eq:M} takes values in $\InjPart$ and that the category $\FinDecGraphs$ embeds in the category $\InjPart$.

\begin{theorem}\label{thm:M_is_embedding}
The functor
\[
\Mb(-,-)\colon\FinDecGraphs\to\InjPart,
\]
is full and faithful.
\end{theorem}
\begin{proof}
First we check that for a morphism in $\DecGraphs$,
\[
(f_G,f_\mathcal{H})\colon (G,\mathcal{H})\to (G',\mathcal{H}'),
\]
the morphism $\Mb(f_G,f_\mathcal{H})$ belongs to $\InjPart$. In fact, let $x\in\ker(\Mb(f_G,f_\mathcal{H}))$ have finite order. Then, by Lemma \ref{lem:1}.\ref{lem:1b}, there exists $v\in V$ such that $x\in H_v$. If $x\neq 1$, then $(x)\in \CR(\{v\})$ is cyclically reduced and $\Mb(f_G,f_\mathcal{H})(x)=1=f_v(x)$. As $f_v$ is injective, we get that $x=1$, a contradiction.

Now, for $(G,\mathcal{H})$ and $(G',\mathcal{H}')$ in $\FinDecGraphs$ and a morphism
\[
f\in \Mor_\InjPart(\Mb(G,\Hc),\Mb(G',\Hc'))
\]
induced by the map of sets
\[
f\colon \Mc(G,\Hc)\to \Mc(G',\Hc'),
\]
we will construct a morphism of decorated graphs
\[
(f_G,f_\mathcal{H})\colon (G,\mathcal{H})\to (G',\mathcal{H}')
\]
such that $\Mb(f_G,f_\mathcal{H})=f$.  Fix a vertex $v\in V$. Then, by Lemma \ref{lem:1}.\ref{lem:1a}, $H_v$ is a non-trivial finite subgroup of $\Mb(G,\mathcal{H})$ and, as $\ker(f)$ has no torsion, the restriction $f_{|H_v}$ is injective. Thus, by Lemma \ref{lem:subgroup_to_subgroup},   $f(H_v)$ is a non-trivial finite subgroup of $\Mb(G',\mathcal{H}')$. By Lemma \ref{lem:1}.\ref{lem:1b}, there exists a unique $v'\in V'$ such that $f(H_v)\subseteq H'_{v'}$ and the group homomorphism
\[
f_{|H_v}\colon H_v \to H'_{v'}
\]
is injective. The correspondence $v\mapsto f_G(v):=v'$ defines a map $f_G\colon V\to V'$ from the vertices of $G$ to those of $G'$. In fact, this map is a graph homomorphism: for every $\{v,w\}\in E$ and every  element $(h_v,h_w)\in H_v\times H_{w}$ with $h_v\neq 1\neq h_w$, $((h_v),(h_w))\in \Db(G,\Hc)$ by Lemma \ref{lem:1}.\ref{lem:1c}. Thus, if we apply $f$, we get that $((f(h_v)),(f(h_w)))\in \Db(G',\Hc')$ and, by Lemma \ref{lem:1}.\ref{lem:1c} again, $\{f_V(v),f_V(w)\}\in E$.

So me may define the morphism $(f_G,f_\mathcal{H})\colon (G,\mathcal{H})\to (G',\mathcal{H}')$ in $\FinDecGraphs$, where $f_G$ is the graph homomorphism constructed above and $f_\Hc=\{f_v\}_{v\in V}$ with
\begin{equation}\label{equ:deffv}
f_v = f_{|H_v}
\end{equation}
for each vertex $v$ of $G$. Let $u=(g_1,\ldots,g_r)\in \CR(C)$ for a clique $C=\{v_1,\ldots,v_n\}$ of $G$ and with with $g_j\in H_{v_{i_j}}$, see Equation \eqref{eq:S}. Note that $(g_i)\in \CR(C)$ for all $i$ and that, as $\Pi(G,\Hc)$ is given by reduced form, $\Pi(G,\Hc)((g_1),\ldots,(g_r))=(g_1,\ldots,g_r)$.  Then we also have that,
\begin{eqnarray*}
f(g_1,\ldots,g_r)&=&f(\Pi(G,\Hc)((g_1),\ldots,(g_r)))\\
&\stackrel{\textrm{(H2)}}=&\Pi(G',\Hc')(\overline{f}((g_1),\ldots,(g_r)))\\
&=&\Pi(G',\Hc')((f(g_1)),\ldots,(f(g_r)))\\
&=&\Pi(G',\Hc')((f_{|H_{v_{i_1}}}(g_1)),\ldots,(f_{|H_{v_{i_r}}}(g_r)))\\
&\stackrel{\eqref{equ:deffv}}=&\Pi(G',\Hc')((f_{v_{i_1}}(g_1)),\ldots,(f_{v_{i_r}}(g_r)))\\
&\stackrel{\eqref{equ:varphi}}=&\Pi(G',\Hc')(\overline{\varphi_{(f_G,f_\mathcal{H})}}((g_1),\ldots,(g_r)))\\
&\stackrel{\textrm{(H2)}}=&\varphi_{(f_G,f_\mathcal{H})}(\Pi(G,\Hc)((g_1),\ldots,(g_r)))\\
&=&\varphi_{(f_G,f_\mathcal{H})}(g_1,\ldots,g_r).
\end{eqnarray*}
Thus $\varphi_{(f_G,f_\mathcal{H})}=f$, $\Mb(f_G,f_\mathcal{H})=f$, and the functor $\Mb(-,-)$ is full. It remains to prove that this functor is faithful but this is straightforward: if we start with a morphism
\[
(f_G,f_\Hc)\in \Mor_\FinDecGraphs((G,\Hc),(G,\Hc'))
\]
and apply the earlier construction to
\[
f=\Mb(f_G,f_\Hc)\in \Mor_\InjPart(\Mb(G,\Hc),\Mb(G',\Hc')),
\]
it is immediate that we recover the morphism $(f_G,f_\Hc)$.
\end{proof}

If we particularise the last result to automorphisms, we get the following exact sequence.

\begin{theorem}\label{thm:aut_of_partial}
Let $(G,\Hc)\in \FinDecGraphs$. Then there exists an exact sequence
\begin{equation}\label{eq:aut_FinDecGraphs}
1\to \prod_{v\in V}\Aut_\Groups(H_v)\to \Aut_\Part\big(\Mb(G,\mathcal{H})\big)\to \Aut_\Graphs(G).
\end{equation}
\end{theorem}
\begin{proof}
Because automorphisms have trivial kernel, we have that
\[
\Aut_\Part\big(\Mb(G,\mathcal{H})\big)=\Aut_\InjPart\big(\Mb(G,\mathcal{H})\big).
\]
Then, by Theorem \ref{thm:M_is_embedding}, we have a bijection
\[
\xymatrix@R=5pt{
\Aut_\FinDecGraphs(G,\Hc)\ar[r]&\Aut_\InjPart\big(\Mb(G,\mathcal{H})\big)\\
(f_G,f_\Hc)\ar@{|->}[r]&\Mb(f_G,f_\Hc).
}
\]
	We define the map $\Aut_\Part\big(\Mb(G,\mathcal{H})\big)\to \Aut_\Graphs(G)$ in the statement by sending an automorphism of the partial group $\Mb(G,\mathcal{H})$ to its corresponding graph automorphism $f_G$. If $f_G=\id_G$ then, for all $v\in V$, $f\vert_{H_v}$ is an automorphism of $H_v$. Finally, any collection $(f_v)_{v\in V}\in \prod_{v\in V}\Aut(H_v)$ induces, via the bijection above, the partial group automorphism,
\[
\Mb(id_G,\{f_v\}_{v\in V})\colon \Mb(G,\mathcal{H})\to\Mb(G,\mathcal{H}),
\]
and the corresponding map $\prod_{v\in V}\Aut(H_v)\to \Aut_\Part\big(\Mb(G,\mathcal{H})\big)$ is injective.
\end{proof}

\begin{remark}\label{rmk:exact_not_short}
Observe that although the sequence \eqref{eq:aut_FinDecGraphs} above is exact, it may not be short exact. For example, let $K_2$ be the complete graph on two vertices, labeled $a$ and $b$, and let $\Hc$ be the collection of finite groups given by $H_a=\Z_2$ and $H_b=\Z_3$. Then $\Aut_\Groups(H_a)=\{1\}$, $\Aut_\Groups(H_b)=\Z_2$ and $\Aut_\Graphs(K_2)=\Z_2$. The cliques in $K_2$ are $\Delta=\{\{a\},\{b\},\{a,b\}\}$ and we have
\begin{align*}
\CR(\{a\})=&\{(a)\}\text{, }\CR(\{b\})=\{(b),(b^2)\}\text{, and }\\
\CR(\{a,b\})=&\{(a),(b),(b^2)\}\cup \{(a,b^{\epsilon_1},\ldots,a,b^{\epsilon_n}),n\geq 0,\epsilon_i\in\{1,2\}\}\cup\\
& \{(b^{\epsilon_1},a,\ldots,b^{\epsilon_n},a),n\geq 0,\epsilon_i\in\{1,2\}\}.
\end{align*}
where we have identified the vertices with the element of order 2 or 3 in the associated decorating group. From here, an easy calculation shows that the only non-trivial automorphism of $\Mb(G,\mathcal{H})$ is given by $a\mapsto a,b\mapsto b^2,b^2\mapsto b$. Hence $\Aut_\Part\big(\Mb(G,\mathcal{H})\big)=\Z_2$ and the sequence \eqref{eq:aut_FinDecGraphs} becomes
\[
1\to \Z_2\to \Z_2\to \Z_2,
\]
which cannot be exact.
\end{remark}

\begin{remark}\label{rmk:section}
A particular case for which the sequence \eqref{eq:aut_FinDecGraphs} is short exact is given by the condition
\[
H_v = H_{v'}\text{ for all $v,v'\in V$.}
\]
Under this condition, assuming the hypotheses in Theorem \ref{thm:aut_of_partial}, and writing $A$ for the common group $H_v$ for any $v\in V$, the map
\[
\xymatrix@R=5pt{
\Aut_\Graphs(G)\ar[r]&\Aut_\Part\big(\Mb(G,\mathcal{H})\big)\overset{\cong}\to \Aut_\FinDecGraphs(G,\Hc)\\
f_G\ar@{|->}[r]&(f_G,\{\id_A\colon A_v\to A_{f_G(v)}\}_{v\in V})
}
\]
is a section of the map $\Aut_\Part\big(\Mb(G,\mathcal{H})\big)\to \Aut_\Graphs(G)$ in \eqref{eq:aut_FinDecGraphs}. Hence the sequence is short exact. In addition, we have the wreath product decomposition

\begin{align*}
\Aut_\Part\big(\Mb(G,\mathcal{H})\big)&\cong \overbrace{\Aut(A)\times\ldots\times \Aut(A)}^{|V|\text{ times}} \rtimes \Aut_\Graphs(G)\\
&\cong \Aut(A) \wr \Aut_\Graphs(G),
\end{align*}
where we have used that every injective endomorphism of a finite group is an automorphism.
\end{remark}

\section{Path partial groups}\label{sec:path_partial}

In this section, we discuss how the algebraic structure of the partial group associated to a decorated undirected graph (Definition \ref{def:DM}) is related to path concatenation in the path algebra corresponding to a quiver. The main difference is that in the path algebra two non-composable paths multiply to $0$, while in our construction such a product does not exist. We first consider the following functor,
\[
\xymatrix@R=5pt{
\Graphs \ar[r] & \FinDecGraphs\\
G=(V,E)\ar@{|->}[r] & (G,\{\Z_2\}_{v\in V}).
}
\]
So a graph is sent to itself decorated with the group $\Z_2$ on each vertex. On morphisms, the graph homomorphism $f_G\colon G\to G'$ is sent to the morphism
\[
(f_G,\{\id_{\Z_2}\colon (\Z_2)_v\to (\Z_2)_{f_G(v)}\}_{v\in V}).
\]
Below we prove that the category $\Graphs$ embeds into the category $\InjPart$.

\begin{lemma}\label{lem:graphs_embeds}
The composite functor
\[
\xymatrix{
\Graphs \ar[r]&\FinDecGraphs\ar[rr]^{\Mb(-,-)}&&\InjPart,
}
\]
is full and faithful.
\end{lemma}
\begin{proof}
Because of Theorem \ref{thm:M_is_embedding}, it is enough to prove that for any two graphs $G$ and $G'$ in $\Graphs$, the map
\[
\xymatrix@R=5pt{
\Mor_\Graphs(G,G')\ar[r] & \Mor_\FinDecGraphs((G,\{\Z_2\}_{v\in V}),(G',\{\Z_2\}_{v\in V}))\\
f_G\ar@{|->}[r]& (f_G,\{\id_{\Z_2}\colon (\Z_2)_v\to (\Z_2)_{f_G(v)}\}_{v\in V})
}
\]
is a bijection, and this is straightforward as the only injective group homomorphism $\Z_2\to \Z_2$ is the identity.
\end{proof}

In the next three paragraphs, we provide an informal explanation of how the construction of path algebras inspired the construction of the path partial group $$\Pb(G)=\Mb(G,\{\Z_2\}_{v\in V})$$ for an undirected graph $G$. To justify this notation recall that, by Equation \eqref{eq:S}, the elements in the alphabet $\Mc=\Mc(G,\{\Z_2\}_{v\in V})$ are the cyclically reduced words $u=(g_1,g_2,\ldots,g_r)\in \CR(C)$ for some clique $C=\{v_1,\ldots,v_n\}$ of $G$, with $g_i\in v_{i_j}$ for $i=1,\ldots,r$. As the elements $g_i$ must be non-trivial elements of $\Z_2$, the cyclically reduced word $u$ is completely determined by the sequence of vertices
\begin{equation}\label{equ:crword_is_sequence_of_vertices}
(v_{i_1},v_{i_2},\ldots,v_{i_r}).
\end{equation}
As all these vertices belong to the clique $C$, this sequence can be thought of as the path $(\{{v_{i_1}},v_{i_2}\},\ldots,\{v_{i_{r-1}},v_{i_r}\})$ in $G$. For $r=1$, we have the path consisting of a single vertex, $(v_{i_1})$. Thus elements in $\Mc$ can be seen as \emph{paths} lying in some clique of $G$ and with $1$ or more vertices.

By Definition \ref{def:PG}.\ref{cond:P3}, there are inverses in every partial group. In particular, the inverse of a vertex should be itself, and this explains the choice of $\Z_2$ in the embedding $\Graphs\to \FinDecGraphs$. Moreover, the inverse of a path should be the reversed path and thus, if $G$ were directed, it should contain the reversed arrow of each arrow. This is why we ask $G$ to be an undirected graph. As described above, we only consider paths in $G$ that are contained in some clique of $G$. The reason for this is that the existence of inverses imply cancellations. For instance, assume that we have two paths in $G$ given by $(a,b)$ and $(b,c)$, where $a,b,c$ are vertices of $G$.  If we multiply them, we would obtain
\[
\Pi((a,b)\circ(b,c))=\Pi(a,b,b,c)=(a,c),
\]
where the two consecutive elements $b$ cancel each other as they correspond to the same element in $\Z_2$. Then $(a,c)$ should be a path in $G$ and the edge $\{a,c\}$ should be contained in $G$. Hence, paths that may be multiplied should have vertices contained in some clique of $G$ and, in addition, multiplication is path concatenation with the extra feature that two consecutive identical vertices produce the empty path.

Finally, we wanted $\Pb(G)$ to have partial group automorphisms ``close'' to the graphs automorphisms of $G$. In particular, the inner automorphism given by conjugation by any vertex should not produce an automorphism of $\Pb(G)$, since inner automorphisms are what prevents the category $\Groups$ to be universal \cite{outer}. This explains why we only consider paths in $G$ which are not closed or, equivalently, words in $\Mc$ that are cyclically reduced.

That the constraints we have imposed are correct is supported by the following results.

\begin{theorem}\label{thm:aut_of_partial_Z2}
Let $G\in Ob(\Graphs)$. Then
\[
\Aut_\Part\big(\Pb(G))\cong\Aut_\Graphs(G).
\]
\end{theorem}
\begin{proof}
Since $\Aut(\Z_2)$ is trivial, the exact sequence \eqref{eq:aut_FinDecGraphs} in Theorem \ref{thm:aut_of_partial} gives rise to an injective homomorphism  $\Psi\colon\Aut_\Part(\Pb(G))\to \Aut_\Graphs(G)$. In addition, the composite functor in Lemma \ref{lem:graphs_embeds} induces a group homomorphism
\[
\Phi\colon \Aut_\Graphs(G) \to \Aut_\Part(\Pb(G)),
\]
which is nothing but the section of \eqref{eq:aut_FinDecGraphs} constructed in Remark \ref{rmk:section}, hence $\Psi\circ\Phi=\id_{\Aut_\Graphs(G)}$. Therefore, they are both isomorphisms inverse to each other.
\end{proof}

\begin{theorem}\label{thm:theorem_A}
The category $\Part$ is universal. Moreover, given an abstract group $H$ there exist infinitely many non isomorphic partial groups $\Mb$ such that $\Aut_\Part(\Mb)\cong H$.
\end{theorem}
\begin{proof}
Let $H$ be an abstract group. According to \cite{Fr39, Gr, Sa}, there exists infinitely many non isomorphic simple graphs $G=(V,E)$, such that $\Aut_\Graphs(G)\cong H$. Then, these non isomorphic graphs $G$ give rise to partial groups $\Pb(G)$ which are not isomorphic by Lemma \ref{lem:graphs_embeds}, and such that $\Aut_\Part(\Pb(G))\cong H$ by Theorem \ref{thm:aut_of_partial_Z2}.
\end{proof}

A rigid object in a category $\C$ is an object $C$ such that $\Aut_\C(C)=\{1\}$. Observe that, while there exists just one non-trivial rigid object in $\Groups$, i.e., $\Z_2$ is the only non-trivial group $K$ such that $\Aut_\Groups(K)=\{1\}$, the situation in $\Part$ is drastically different: Theorem \ref{thm:theorem_A} ensures that there exist infinitely many non-trivial rigid partial groups. 

\section{The graph of maximal locally finite subgroups of a partial group}\label{sec:maximalsub}

A group $H$ is called \emph{locally finite} if every finitely generated subgroup of $H$ is finite. In this section we construct a graph out of the maximal locally finite subgroups of a given partial group. Recall that a subgroup of a partial group is a partial subgroup which is itself a group, see Definition \ref{def:partial_subgroup_and_subgroup}. This will allow us to show that, given a simple undirected graph $G=(V,E)$ and a collection of non-trivial locally finite groups $\mathcal{H}:=\{H_v\}_{v\in V}$, $G$ can be fully recovered (up to isomorphism) from the algebraic structure of $\Mb(G,\Hc)$.

We first prove that in a partial group there is at least a maximal locally finite subgroup.

\begin{lemma}
Let $\Mb=\left(\Mc,\Db,\Pi,(-)^{-1}\right)$ be a partial group.
Then the poset of the locally finite subgroups of $\Mb$ ordered by inclusion has a maximal element.
\end{lemma}

\begin{proof}
Let $\X$ be the set of locally finite subgroups of $\Mb$. Notice that $\X$ is not empty since $\{1\}$ is a locally finite subgroup of $\Mb$. If $H_1\leq H_2\leq \cdots$ is an increasing chain of locally finite subgroups of $\Mb$, consider the subgroup $H_\infty=\bigcup_{i\in \N} H_i\leq \Mc$. By construction, every finitely generated subgroup of $H_\infty$ lies in some $H_i$ for $i$ big enough. Therefore, $H_\infty$ is a locally finite subgroup of $\Mb$ and thus an upper bound for the given chain of elements of $\X$. Applying  Zorn's Lemma, we obtain that $\X$ has a maximal element.
\end{proof}

\begin{definition}\label{def:maxsubgraph}
Let $\Mb=\left(\Mc,\Db,\Pi,(-)^{-1}\right)$.
The \emph{maximal locally finite subgroup graph} of $\Mb$ is the graph $\MS(\Mb)$ with vertices the set of non-trivial maximal locally finite subgroups of $\Mb$ and declaring two \emph{distinct} vertices $H_1$ and $H_2$ to be adjacent if and only if there exist $h_1\in H_1^*$ and $h_2\in H_2^*$ such that $(h_1,h_2)\in \Db $.
\end{definition}

\begin{remark}\label{rmk:MS_well_defined}
Notice that $\MS(\Mb)$ is a well-defined simple undirected graph. Indeed, given two different non-trivial maximal locally finite subgroups $H_1$ and $H_2$ of $\Mb$ and two elements $h_1\in H_1^*$ and $h_2\in H_2^*$, we have
\[
(h_1,h_2)\in \Db\text{ if and only if }(h_2^{-1},h_1^{-1})\in \Db.
\]
Thus, the adjacency relation considered is symmetric.
\end{remark}

If $\Mb$ has no non-trivial subgroups, then the graph $\MS(\Mb)$ contains only the trivial subgroup, and Example \ref{ex:freepargp} provides an example of such situation. In Definition \ref{def:maxsubgraph}, we look at locally finite subgroups instead of at subgroups because in a path partial group every edge give rise to a copy of $\Z$. This is demonstrated in Example \ref{example:K2_K3} below.

\begin{example}\label{example:K2_K3}
Let $K_2$ be the complete graph on two vertices, labeled $a$ and $b$. Then the Hasse diagram of the poset of cliques $\Delta$ of $K_2$ is the following,
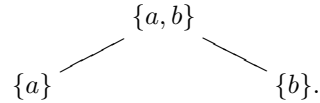
\[
\xymatrix@R=10pt{
&\{a,b\}\\
\{a\}\ar@{-}[ru]&&\{b\}.\ar@{-}[lu]
}
\]
Recall that $\Pb(K_2)=\left(\Mc,\Db,\Pi,(-)^{-1}\right)$ is the colimit in the category $\Part$ of the partial groups
\[
\Mb(C)=(\CR(C),\Db(C),\Pi(C),(-)_C^{-1})\]
for $C\in \Delta$. Notice that
\begin{align*}
&\CR(\{a\})=\{(a)\}\text{, }\CR(\{b\})=\{(b)\}\text{, and }\\
&\CR(\{a,b\})=\{(a),(b)\}\cup \{\underbrace{(a,b,\ldots,a,b)}_{2n},n\geq 0\}\cup \{\underbrace{(b,a,\ldots,b,a)}_{2n},n\geq 0\},
\end{align*}
where we have identified the vertices with the element of order 2 in the associated decorating group.
The maximal subgroups in $\Pb(K_2)$ are $\langle (a)\rangle\cong \Z_2$, $\langle (b)\rangle\cong \Z_2$, and $\langle(a,b)\rangle\cong \Z$. Only two of these three subgroups are locally finite and it is straightforward that $\MS(\Pb(K_2))$ is the complete graph on two vertices.
\end{example}

There are alternative definitions for the \emph{maximal locally finite subgroup graph} associated to a partial group that make sense but that we do not employ here. For instance, one could consider the same set of vertices as that in Definition \ref{def:maxsubgraph} but endowed with the ``stronger'' notion of adjacency described as follows: $H_1$ and $H_2$ are adjacent if and only if \emph{for all} $h_1\in H_1^*$ and \emph{for all} $h_2\in H_2^*$, $(h_1,h_2)\in \Db$. In the particular case of path partial groups both definitions agree, but in general non-isomorphic graphs are obtained as the next example shows.

\begin{example}\label{example:D_8}
Consider the dihedral group of size $8$,
\[
D_8=\langle x,t \;|\; x^4=1,t^2=1,txt=x\rangle,
\]
its center, and its two Klein four-groups,
\[
Z=Z(D_8)=\langle x^2 \rangle \text{, }V=\langle x^2,t\rangle \text{ and }V'=\langle x^2,tx\rangle.
\]
Define the partial group $\Mb=\left(\Mc,\Db,\Pi,(-)^{-1}\right)$ as the colimit in the category $\Part$ (see Section \ref{sec:partials}) of the diagram given by inclusions among these three subgroups,
\[
\Mb= \mathop{colim}_{\Part} (V\leftarrow Z\rightarrow V').
\]
Then the set of maximal locally finite subgroups of $\Mb$ is exactly $\{V,V'\}$. Moreover, the word $(x^2,x^2)$ belongs to $\Db$  and hence the graph $\MS(\Mb)$ is connected. Nevertheless, the word $(t,tx)$ does not belong to $\Db$ and hence the graph defined using the ``stronger'' condition in the paragraph above is not connected. In particular, the partial group $\Mb$ cannot be isomorphic to a partial group of the form $\Mb(G,\Hc)$ for any pair $(G,\Hc)\in \DecGraphs$.
\end{example}

\begin{lemma}\label{lem:locmassub}
Let  $G=(V,E)$ be a simple undirected graph and $\mathcal{H}:=\{H_v\}_{v\in V}$ be a collection of non-trivial locally finite groups. Then, for every non-trivial locally finite subgroup $H$ of $\Mb(G,\Hc)$, there exist a unique $v\in V$ such that $H\leq H_v$. In particular, the set of maximal locally finite subgroups of $\Mb(G,\mathcal{H})$ is $\Hc$.
\end{lemma}

\begin{proof}
Let $H\leq \Mb(G,\mathcal{H})$ be a non-trivial locally finite subgroup. Then $H$ contains at least a non-trivial finite group $K$ and, by Lemma \ref{lem:1}.\ref{lem:1b}, there exists a vertex $v\in V$ such that $K\leq H_v$. Hence $1\neq K \leq H_v\cap H$. Now assume that there are vertices $v_1$ and $v_2$ in $V$ such that $H_{v_1}\cap H$ and $H_{v_2}\cap H$ are non-trivial. Then we may choose $h_1\in (H_{v_1}\cap H)^*$ and $h_2\in (H_{v_2}\cap H)^*$ and consider the finite subgroup $\langle h_1,h_2\rangle\leq H$ of $\Mb(G,\mathcal{H})$. Then, by Lemma \ref{lem:1}.\ref{lem:1b} again, there exists a unique vertex $v\in V$ such that $\langle h_1,h_2\rangle\leq H_v$. Since $\langle h_1\rangle\leq H_{v_1}$ is a non-trivial finite subgroup which is contained in both $H_{v_1}$ and $H_v$, the uniqueness in Lemma \ref{lem:1}.\ref{lem:1b} gives that $v_1=v$. The analogous argument applied to $h_2$ shows that $v_2=v$ and hence $v_1=v_2$. This proves the first part of the lemma.

For the second part, notice that for every $v\in V$, $H_v$ is a non-trivial locally finite subgroup of $\Mb(G,\mathcal{H})$ by Lemma \ref{lem:1}.\ref{lem:1a}. Therefore, by what we just proved in the paragraph above, the $H_v$'s are maximal locally finite subgroups of $\Mb(G,\mathcal{H})$ and every maximal locally finite subgroup of $\Mb(G,\mathcal{H})$ is of this form
\end{proof}

\begin{theorem}\label{thm:maximalsub}
Let $G=(V,E)$ be a simple undirected graph and $\mathcal{H}:=\{H_v\}_{v\in V}$ be a collection of non-trivial locally finite groups.
Then $\MS\big(\Mb(G,\Hc)\big)\cong G$.
\end{theorem}

\begin{proof}
By Lemma \ref{lem:locmassub}, the maximal locally finite subgroups of $\Mb(G,\Hc)$ are exactly the elements of $\Hc$. Moreover, thanks to Lemma \ref{lem:1}.\ref{lem:1c}, $H_v$ and $H_{v'}$ are adjacent in $\MS\big(\Mb(G,\Hc)\big)$ if and only $\{v,v'\}\in E$. Hence the application which sends a vertex $v\in V$ to $H_v$ defines an isomorphism of graphs between the graph $G$ and the graph $\MS\big(\Mb(G,\Hc)\big)$.
\end{proof}

Notice that by also considering the maximal locally finite subgroups themselves, one recovers all the data, i.e., the graph and the decorating groups. Theorem \ref{thm:maximalsub} lays the foundations to start the study of Problem \ref{problem:proposed}: as the isomorphism type of $\Mb(G,\Hc)$ in $\Part$ determines the isomorphism type of $G$ in $\Graphs$, the graph theoretical properties of $G$ should be reflected on the algebraic properties of the partial group $\Mb(G,\Hc)$ and vice versa.

\section{Declarations}

The authors have no conflicts of interest to declare that are relevant to the content of this article.

\bibliographystyle{abbrv}
\bibliography{biblio} 

\end{document}